\documentclass[11pt,article]{amsart}
\usepackage{amssymb}
\usepackage{amsfonts}
\usepackage{amsbsy}
\usepackage{latexsym}
\usepackage{graphicx}
\usepackage{amsmath}

\newtheorem{theorem}{Theorem}[section]
\newtheorem{lemma}[theorem]{Lemma}
\newtheorem{proposition}[theorem]{Proposition}

\theoremstyle{definition}
\newtheorem{definition}[theorem]{Definition}
\newtheorem{example}[theorem]{Example}

\theoremstyle{remark}

\numberwithin{equation}{section}

\newcommand\N{\mathbb{N}}
\newcommand\R{\mathbb{R}}
\newcommand\T{\mathbb{T}}
\newcommand\C{\mathbb{C}}
\newcommand\J{\mathbb{J}}

\newcommand\Z{\mathbb{Z}}

\newcommand\cH{\mathcal{H}}
\newcommand\cK{\mathcal{K}}
\newcommand\cG{\mathcal{G}}
\newcommand\cB{\mathcal{B}}
\newcommand\cM{\mathcal{M}}

\def\sideremark#1{\ifvmode\leavevmode\fi\vadjust{\vbox
to0pt{\vss \hbox to 0pt{\hskip\hsize\hskip1em
\vbox{\hsize2cm\tiny\raggedright\pretolerance10000
\noindent#1\hfill}\hss}\vbox to8pt{\vfil}\vss}}}

\begin{document}

\title[Dilations for Systems of Imprimitivity acting on Banach Spaces]{Dilations for Systems of Imprimitivity acting on Banach Spaces}

\author{Deguang Han}
\address{Department of Mathematics, University of Central
Florida, Orlando, USA} \email{deguang.han@ucf.edu}

\author{David R. Larson}
\address{Department of Mathematics, Texas A\&M University, College Station, USA}
\email{larson@math.tamu.edu}

\author{Bei Liu}
\address{Department of Mathematics, Tianjin University of Technology, Tianjin, China}
\email{beiliu1101@gmail.com}

\author{Rui Liu}
\address{Department of Mathematics and LPMC, Nankai University, Tianjin, China}
\email{ruiliu@nankai.edu.cn}

\begin{abstract} Motivated by  a general dilation theory for operator-valued measures, framings and bounded linear maps on operator algebras, we consider the dilation theory of the above objects with special structures. We show that every operator-valued system of imprimitivity
has a dilation to a probability spectral system of imprimitivity acting on a Banach space. This completely generalizes a well-kown result which states that every frame representation
of a countable group on a Hilbert space is unitarily equivalent to a subrepresentation of the left regular representation of the group.
The dilated space in general can not be taken as a Hilbert space. However, it can be taken as a Hilbert space for
positive operator valued systems of imprimitivity. We also prove that isometric group representation induced framings on a Banach space can be dilated to unconditional bases with the same structure for a larger Banach space This extends several known results on the dilations of frames induced by unitary group representations on Hilbert spaces.

\end{abstract}

\thanks{The authors were all participants in the NSF funded Workshop
in Analysis and Probability at Texas A\&M University. Deguang Han acknowledges partial support by NSF grant DMS-1106934. Bei Liu and Rui Liu  both are supported by NSFC grants 11201336 and 11001134}

\date{}



\maketitle

\section{Introduction}

In a recent paper
\cite{HLLL}  we developed a general dilation theory for operator-valued
measures, and for bounded linear maps on von Neumann algebras. It is well known 
that, if $A$ is a $C^*$-algebra with unit and $\phi:A\to B(H)$
is a completely bounded map, then there exists a Hilbert space $K$,
a $*$-homomorphism $\pi:A\to B(K)$, and bounded operators $V_i:H\to K, i=1,2,$
with $\|\phi\|_{\mathrm{cb}}=\|V_1\|\cdot\|V_2\|$ such that
\[\phi(a)=V_1^* \pi(a)V_2\] for all $a\in A$ (see Theorem 8.4 in \cite{Pa}). 
Thus, a bounded linear map from a unital $C^*$-algebra into $B(H)$ has a Hilbert space
dilation to a $*$-homomorphism if and only if the mapping is completely bounded
(The necessary part follows from the complete boundedness of $*$-homomorphisms).
The theory tells us that even if the bounded map is not
completely bounded it still has a Banach space dilation to a
homomorphism, and any operator valued measure has a Banach space dilation.
These results can be viewed as generalizations of the known result of Casazza, Han
and Larson in \cite{CHL} that arbitrary framings have Banach space dilations, and also
the known result that completely bounded maps have Hilbert space
dilations. This paper continues the investigation of the dilation theory for operator valued measures and bounded linear
maps between operator algebras. Our focus will be on the dilation theory for structured
operator-valued measures associated with systems of imprimitivity.
This is partially motivated by the dilation results for discrete
structured frames (cf. \cite{GH, GH2, Han1, Han2, HL}), and our new
dilation theorems discussed in \cite{HLLL, HLLL-CM, LS}.

The concept of a
system of imprimitivity was introduced by Mackey \cite{M, Mac-book}  for
his theory of induced representations of locally compact groups. It is used in algebra and analysis in the theory of
group representations. The theory of systems of imprimitivity includes both
the finite dimensional case and the infinite dimensional case. In Mackey's theory,  every
infinite dimensional system of imprimitivity is affiliated with a
projection-valued measure on a Hilbert space. Motivated by the Banach space dilation
aspects of operator valued measures developed in \cite{HLLL}, we develop a dilation theory
of operator valued systems of imprimitivity based on semi-groups acting on both Hilbert and Banach spaces. In section 3, we will prove that
every operator-valued system of imprimitivity
on $(S, \Sigma)$ can be dilated to a probability spectral system of imprimitivity, where $\Sigma$ is a $\sigma$-field of subsets of a set $\Omega$ and $S$
a sub-semigroup of a $\Sigma$-measurable group $G$. Even in the case that the projective operator-valued system of imprimitivity
on $(S, \Sigma)$ is based on a Hilbert space, the dilated probability projective spectral system of imprimitivity in general is based on a Banach space. This Banach
space restriction seems necessary because there is an example of a Hilbert space based operator-valued measure which has a Banach space dilation but not a
Hilbert space dilation. (See Theorem E and the subsequent discussion in the introduction of \cite{HLLL}).
In the special case of
positive operator valued systems of imprimitivity acting on a Hilbert space we show that they can be dilated to (orthogonal) projection valued systems of imprimitivity. We remark that the proof of our main dilation theorem (Theorem 3.1) relies on the existence of a minimal dilation space and a minimal dilation norm that were introduced in \cite{HLLL}.

Related to the dilation of group representation frame generators to wandering vector or more generally Riesz vectors the first named author of this paper proved a
geometric structural theorem for the dilations of dual frame pairs induced by a group representation acting on a Hilbert space. The proof of this theorem (especially
for the subspace dual frame pair case) very much involves techniques from the theory of group von Neumann algebras \cite{Han1, Han2}. We extend the above result in section 4 to framings (a concept
that  generalizes dual frame pairs and captures the Banach space nature of the frame dilation theory) that are induced by isometric representations of a countable group on Banach spaces. We prove that any such a framing has
an unconditional basis dilation of the same structure. Moreover, the dilation Banach space can be explicitly constructed.

\section{Preliminaries and Examples}

The concept of infinite dimensional systems of imprimitivity in
Hilbert spaces can be generalized to Banach space cases. The ingredients are
 a semigroup $S$, a
$\sigma$-field $\Sigma$ of subsets of a set $\Omega$, and a
measurable semigroup action \[S\times\Sigma\to \Sigma, \ \
(s,E)\mapsto sE,\] that satisfies that for all $E\in\Sigma$ and
$s,t\in S$:
\begin{enumerate}
\item $s(t E)=(st)E$
\item $e E=E$, $s\emptyset=\emptyset$, $s\Omega=\Omega$
\item $s(\bigcup_{n=1}^\infty
E_n)=\bigcup_{n=1}^\infty s E_n$
\item
$s(\bigcap_{n=1}^\infty E_n)= \bigcap_{n=1}^\infty s E_n$.
\end{enumerate}
We refer to this as a \emph{$\Sigma$-measurable $S$-space}, denoted
by $(S,\Sigma)$,  for example,$(\N,\R^+)$.


In this paper the term \emph{semigroup} will signify a semigroup with unit.
We shall usually write the operation multiplicatively and denote the
unit by $e$. In particular, a subsemigroup $S$ of a $\Sigma$-measurable
group $G$ is a $\Sigma$-measurable $S$-space with the relative measure.

A \emph{multiplier} on a semigroup $S$ is a function $\omega$, from
the Cartesian product $S\times S$ to the unit circle $\T$ in the
complex plane $\C$, such that for all elements $s, t$ and $u$ of
$S$:
\begin{enumerate}
  \item $\omega(e,s)=\omega(s,e)=1$;
  \item $\omega(s,t)\omega(st,u)=\omega(s,tu)\omega(t,u)$.
\end{enumerate}
If $S$ is a group, then if follows from (1) and (2) that for all
$s\in S$ , we have $\omega(s,s^{-1})=\omega(s^{-1},s).$ If a
multiplier $\omega$ satisfies
$\omega(s,s^{-1})=\omega(s^{-1},s)=1$ for all $s\in S$, we say that  $\omega$ is
\emph{symmetric}.

We generalize the concept of projective isometric representations (see \cite{Mu}) from
Hilbert spaces to Banach spaces. Let $S$ be a semigroup and $X$ a Banach space. A \emph{projective
isometric representation} of $S$ on $X$ is a map, $W:s\mapsto W_s$,
from $S$ to $B(X)$ having the following properties for all elements
$s, t\in S$:
\begin{enumerate}
  \item $W_s$ is an isometry and $W_e=1$;
  \item $W_sW_t=\omega(s,t)W_{st}$, where $\omega(s,t)$ are scalars of unit modulus..
\end{enumerate}
It follows from the equations that the function,
$\omega:(s,t)\mapsto\omega(s,t),$ is a multiplier on $S$, it is called
the multiplier \emph{associated to} $W$.

If all the $W_s$ are surjective isometries, we say that $W$ is a
\emph{projective isometric isomorphism representation}. In this
case, if the associated multiplier $\omega$ is symmetric, then for
all $s\in S$ $W_{s^{-1}}=W_s^{-1}$.

A projective isometric representation or projective isometric
isomorphism representation with $\omega$ an associated multiplier
will sometimes be referred to as an \emph{isometric
$\omega$-representation} or, respectively, an \emph{isometric
isomorphism $\omega$-representation}.

We remark that if $S$ is a group, then a projective
isometric representation is automatically a projective
isometric  isomorphism representation.

Let $X$ and $Y$ be Banach spaces, and let $(\Omega,\Sigma)$ be a
measurable space. A \emph{$B(X,Y)$-valued measure on $\Omega$} is a
map $E:\Sigma\to B(X,Y)$ that is countably additive in the weak
operator topology; that is, if $\{B_i\}$ is a disjoint countable
collection of members of $\Sigma$ with union $B$, then
\begin{equation*}
y^{*}(E(B)x)=\sum_i y^{*} (E(B_i)x)
\end{equation*}
for all $x\in X$ and $y^*\in Y^*.$
The Orlicz-Pettis theorem  states that weak unconditional
convergence and norm unconditional convergence of a series are the
same in every Banach space (c.f.\cite{DJT}). Thus  we have  that
$\sum_i E(B_i)x$ weakly unconditionally converges to $E(B)x$ if
and only if $\sum_i E(B_i)x$ strongly unconditionally converges to
$E(B)x.$ So it is equivalent to saying that
$E$ is strongly countably additive, that is, if $\{B_i\}$ is a
disjoint countable collection of members of $\Sigma$ with union
$B$, then
\begin{equation*}
E(B)x=\sum_i E(B_i)x
\end{equation*}
for all $x\in X$.

A $B(X)$-valued measure $E$ on $(\Omega,\Sigma)$ is called:
\begin{enumerate}
\item[(i)] an operator-valued probability measure if
$E(\Omega)=\mathrm{I}_X,$

\item[(ii)]a projection-valued measure if $E(B)$ is a projection
on $X$ for all $B\in\Sigma$,

\item[(iii)] a spectral operator-valued measure if for all $A,
B\in \Sigma, E(A\cap B)=E(A)\cdot E(B)$ (we will also use the term
idempotent-valued measure to mean a spectral-valued measure.)
\end{enumerate}
It is an elementary fact that a $B(X)$-valued measure which is a
projection-valued measure is always a spectral-valued measure (c.f.
\cite{HLLL, Pa}). Note that spectral
operator-valued measures are clearly projection-valued measures. So (ii) and (iii) are equivalent.

Let $\mathcal {G}$ be a locally compact Hausdorff topological group
acting on a measurable space $(\Omega,\Sigma)$. A  (orthogonal) \emph{projection-valued system of
imprimitivity} (respectively, \emph{positive operator-valued system of
imprimitivity})  based on $(\mathcal {G},(\Omega,\Sigma))$ consists of
a separable Hilbert space $\cH$ and a pair consisting of a
strongly-continuous unitary representation $U:g\mapsto U_g$ of
$\mathcal {G}$ on $\cH$, and a (orthogonal)  projection-valued measure (respectively, a positive operator-valued measure) $\pi$ on
the measurable subsets of $\Omega$ with values in the projections (respectively, positive operators) on
$\cH$, which satisfy
\[U_g \,\pi(E) U_{g^{-1}}=\pi(g\cdot E) \quad \mbox{ for all }
g\in \mathcal {G} \mbox{ and } E\in \Sigma.\]

\begin{example}\label{ex:1} Let $\mu$ be a left Haar measure on the Borel subsets
$\mathcal{B}$ of a locally compact Hausdorff topological group
$\cG$, and  $U:\cG\to B(\cH)$ be a strongly-continuous unitary representation
of $\cG$ on a separable Hilbert space $\cH$. Then $f\in \cH$ is
called a \emph{Bessel vector} if there is a constant $C>0$ such that
\[\int_\cG |\langle x, U_g f \rangle_\cH|^2 d \mu(g)\le C \|x\|^2
\quad \mbox{ for all } x\in \cH.\]
For a Bessel vector $f\in \cH$, we define
$\varpi_f:\mathcal{B}\rightarrow B(\cH)$ by
\[\varpi_f(E):=\int_{E} U_g f \otimes U_g f\,
 d\mu(g) \quad \mbox{ for all } E\in \mathcal{B},\]
that is,
\[\langle\varpi_f(E)x,y\rangle_\cH:=\int_{E}\langle x,
U_g f\rangle_\cH \langle U_g f, y\rangle_\cH\, d\mu(g)
\quad \mbox{ for all } x,y\in \cH,\]
or equivalently,
\[\varpi_f(E)x:=\int_{E}\langle x,
U_g f\rangle_\cH U_g f \,d\mu(g)
\quad \mbox{ for all } x\in \cH.\]
Then $\varpi_f$ is a
positive-operator valued measure on $(\cG,\mathcal{B})$.


Moreover, $(U, \varpi_f)$ is a positive operator-valued system of
imprimitivity.
Indeed, for any $g\in\cG$, $E\in\cB$, and $x,y\in\cH$,
we have
\begin{eqnarray*}
\langle U_g\varpi_f(E)U_{g^{-1}}x, y \rangle &=& \langle
\varpi_f(E)U_{g^{-1}}x, U_{g^{-1}} y \rangle\\
&=&\int_{E} \langle U_{g^{-1}}x, U_{g'}f\rangle\cdot
\langle U_{g'}f, U_{g^{-1}} y\rangle \,d\mu(g')\\
&=&\int_{E} \langle x, U_{g}U_{g'}f\rangle\cdot \langle
U_g U_{g'}f, y\rangle \,d\mu(g')\\
&=&\int_{E} \langle x, U_{g\cdot g'}f\rangle\cdot \langle
U_{g\cdot g'}f, y\rangle \,d\mu(g')\\
&=&\int_{g\cdot E} \langle x, U_{g'}f\rangle\cdot \langle
U_{g'}f, y\rangle \,d\mu(g')\\
&=&\langle \varpi_f(g\cdot E)x, y \rangle.
\end{eqnarray*}
Thus  $U_g\varpi_f(E)U_{g^{-1}}=\varpi_f(g\cdot E)$,
and so $(U, \varpi_f)$ is a positive operator-valued system of
imprimitivity of
$(\cG,\mathcal{B})$ on $\cH$.

\end{example}

The following example is the continuous wavelet transform.

\begin{example}\label{ex:H64}
Let $\Omega=\{(a,b):a>0,\, b\in\mathbb{R}\}$. The action on $\Omega$ is defined by
$(a,b)(s,t)=(as,b+at).$ For any $(a,b)\in \Omega$ and $E\in \mathcal{B},$ we have
\[\iint\limits_{(a,b)E}1\frac{dsdt}{s^{2}}=\iint\limits_{E}1\frac{dsdt}{s^{2}}\]
The left Haar measure $\mu$ on the Borel subsets $\mathcal{B}$ of $\Omega$ is $dsdt/s^2$. Let $\cH=L^2(\mathbb{R})$. For $(a,b)\in \Omega,$
define
\[U_{a,b}:L^2(\mathbb{R})\to L^2(\mathbb{R}), \quad U_{a,b}(f)=a^{-1/2}f(\frac{x-b}{a}).\]
Then $U$ is a strongly continuous unitary representation of $\Omega$ on the separable Hilbert space $L^2(\mathbb{R})$. From the wavelet theory,
we have the following identity for the continuous wavelet transform
(\cite[Proposition 2.4.1]{D}):

For all $f,g\in L^{2}(\mathbb{R})$, we have
\begin{equation*}\label{eq:Hb2}
\iint\limits_{\Omega}\langle f,U_{a,b}(h)\rangle\overline{\langle g,U_{a,b}(h)\rangle}\frac{dadb}{a^{2}} =C_{h}\langle f,g\rangle,
\end{equation*}
Where $h\in L^{2}(\mathbb{R})$ satisfies
\begin{equation}\label{eq:Hb21}
C_{h}=\int^{+\infty}_0\frac{1}{|\omega|}|\hat{h}(\omega)|^{2}d\omega\neq 0.
\end{equation}
 Thus for any $h\in L^{2}(\mathbb{R})$ satisfies (\ref{eq:Hb21}), we
 have
\begin{equation*}\label{eq:Hb3}
\iint\limits_{\Omega}\left|\langle f,U_{a,b}(h)\rangle\right|^2\frac{dadb}{a^{2}} =C_{h}\|f\|^2.
\end{equation*}  Define
\[\varpi_h(E)(f)=\iint\limits_{E}\langle f,U_{a,b}(h)\rangle U_{a,b}(h)
\frac{dadb}{a^{2}}.\] Then $\varpi_h$ is a positive-operator-valued
measure, and so $(U, \varpi_h)$ is a positive operator-valued system of
imprimitivity of
$(\Omega,\mathcal{B})$ on $L^2(\mathbb{R})$.
\end{example}

\section{Dilation of operator-valued systems
of imprimitivity}

Let $S$ be a semigroup acting on a $\sigma$-field $\Sigma$ of
subsets of a set $\Omega$. A \emph{projective isometric operator-valued  system
of imprimitivity} based on $(S,\Sigma)$ consists of a Banach space
$X$ and a pair consisting of:
\begin{enumerate}
\item A projective isometric representation $W:s\mapsto W_s$ of $S$ on
$X$;
\item An operator-valued measure $\varphi$ on
$\Sigma$ with values in the operators on $X$,
\end{enumerate}
satisfying that for all $s\in S$ and $E\in\Sigma$ \[W_s
\varphi(E)=\varphi(s E) W_s.\] A projective isometric operator-valued system
of imprimitivity $(W,\varphi)$ is denoted \emph{probability} if $W$
is a probability operator-valued measure.

A \emph{projective isometric spectral system of imprimitivity} based on $(S,
\Sigma)$ consists of $X$ and a pair consisting of:
\begin{enumerate}
\item A projective isometric representation $W:s\mapsto W_s$ of $S$ on $X$;
\item A spectral measure $\rho$ on
$\Sigma$ with values in the projections on $X$,
\end{enumerate}
satisfying that for all $s\in S$ and $E\in\Sigma$ \[W_s
\rho(E)=\rho(s E) W_s.\] A projective isometric operator-valued system of
imprimitivity or projective isometric  spectral system of imprimitivity with
$\omega$ an associated multiplier will sometimes be referred to as
an \emph{operator-valued $\omega$-system of imprimitivity} or,
respectively, an \emph{spectral $\omega$-system of imprimitivity}.

\begin{example}For some structured frames that include Gabor frames, recall that a projective unitary
 representation $\pi$ for a countable group $G$ is a mapping $g\to \pi(g)$ from $G$ into the set of unitary
operators on a Hilbert space $H$ such that $\pi(g)\pi(h)=\mu(g,h)\pi(gh)$ for all
$g,h\in G$, where $\mu(g,h)$ belongs to the circle group $\mathbb{T}$. The mapping
$(g,h)\to\mu(g,h)$ is then called a multiplier of $\pi$. The image of a projective unitary
representation is also called a group-like unitary system. If a projective representation
$\pi$ on a Hilbert space $H$ admits a frame vector $\xi$, i.e., $\{\pi(g)\xi\}_{g\in G}$
is a frame for $H$, then $\pi$ is called a frame representation. From the proof in Example \ref{ex:1}, it is
easy to see that $\{\pi(g)\xi\}_{g\in G}$ induces a projective operator-valued isometric system
of imprimitivity.
\end{example}


Let $(W,\varphi)$ be an operator-valued $\omega$-system of
imprimitivity of $(S, \Sigma)$ on a Banach space $X$ and $(V,\rho)$
a spectral $\omega$-system of imprimitivity of $(S, \Sigma)$ on a
Banach space $Z.$ Then $(V,\rho)$ is said to be a \emph{dilation} of
$(W,\varphi)$ if there are bounded operators $Q:Z\to X$ and $T:X\to
Z$ such that for all $s\in S$ and $E\in\Sigma$:
\begin{enumerate}
  \item $\varphi(E)=Q\rho(E)T$;
  \item $Q V_s=W_s Q$;
  \item $V_s T=T W_s$.
\end{enumerate}
In this case, $(V,\rho,Q,T)$ is called a \emph{dilation system} of
$(W,\varphi)$. The representation of $S$ on $X$ can be viewed as a subrepresentation
of the representation of $S$ on $Z$.

The following is the first main result of this paper. It says in particular that a representation of a group on a Banach space which is suitably affiliated with an operator-valued
probability measure is a subrepresentation of one which is affiliated with a projection-valued probability measure. This completely generalizes a theorem from \cite{HL} which states that every frame representation
of a countable group on a Hilbert space is unitarily equivalent to a subrepresentation of the left regular representation of the group.

\begin{theorem} \label{main-thm1} Let $\Sigma$ be a $\sigma$-field of subsets of a set $\Omega$ and $S$
a $\Sigma$-measurable semigroup. Then every projective isometric operator-valued system of imprimitivity
of $(S,\Sigma)$ can be dilated to a probability projective  isometrc spectral system of imprimitivity.
\end{theorem}

Since the proof is lengthy and technical, we divide it into several lemmas and propositions.

\begin{lemma}\label{le:446}
Let $(W,\rho)$ be a projective isometric spectral  system of imprimitivity of
$(S,\Sigma)$ on a Banach space $X$. Then
\begin{enumerate}
  \item $\rho(\Omega)X$ is an invariant subspace of $W$;
  \item The restriction $(W|_{\rho(\Omega)X}, \rho|_{\rho(\Omega)X})$ is a
probability projective  isometric spectral system of imprimitivity of
$(S,\Sigma)$ on $\rho(\Omega)X$.
\end{enumerate}
\end{lemma}
\begin{proof} Since $\rho(\Omega)$ is a projection, $\rho(\Omega)X$
is a Banach space. For all $s\in S$,
\[W_s\rho(\Omega)X=\rho(s\Omega)W_s X=\rho(\Omega\cap s\Omega)W_s X
=\rho(\Omega)\rho(s\Omega)W_s X\subset\rho(\Omega)X.\] It follows
that $\rho(\Omega)X$ is an invariant subspace of $W$, then
$W|_{\rho(\Omega)X}$, denoted by $\widehat{W}$, is a projective
isometric representation of $S$ on $\rho(\Omega)X$. It is clear that
$\rho(\Omega)X$ is also an invariant subspace of $\rho$, and
$\rho|_{\rho(\Omega)X}$, denoted by $\hat{\rho}$, is a probability
spectral
measure on $\rho(\Omega)X$. 
%
For all $s\in S$ and $E\in\Sigma$ we have
\[\widehat{W}_s\hat{\rho}(E)=W_s\rho(E)|_{\rho(\Omega)X}=
\rho(s E)W_s|_{\rho(\Omega)X}=\hat{\rho}(s E)\widehat{W}_s.\]
Thus, $(\widehat{W}, \hat{\rho})$ is a probability projective isometric
spectral system of imprimitivity.
\end{proof}

%

\begin{lemma} The subspace $\overline{Q(Z)}$ is invariant subspace under
$(W,\varphi)$, and the restriction $(W|_{\overline{Q(Z)}},
\varphi|_{\overline{Q(Z)}})$ is a projective isometric operator-valued system
of imprimitivity of $(S, \Sigma)$ on $\overline{Q(Z)}$.
\end{lemma}
\begin{proof} It is sufficient to prove that $Q(Z)$ is an invariant subspace of
$(W,\varphi)$. For all $s\in S$, $W_sQ(Z)=QV_s(Z)\subset Q(Z)$. For
all $E\in\Sigma$ we have \[\varphi(E)Q(Z)=Q\rho(E)TQ(Z)\subseteq
Q(Z).\] Thus, $(W|_{\overline{Q(Z)}}, \varphi|_{\overline{Q(Z)}})$
is a projective isometric  operator-valued system of imprimitivity.
\end{proof}



\begin{lemma}\label{pr:448}
Let $(W,\varphi)$ be a projective isometric operator-valued system of
imprimitivity of $(S, \Sigma)$ on $X$. If $(V,\rho,Q,T)$ is a
dilation system of $(W,\varphi)$ on $Z$, then the restriction
$$(V|_{\rho(\Omega)Z}, \rho|_{\rho(\Omega)Z}, Q|_{\rho(\Omega)Z},
\rho(\Omega)T)$$ is a probability dilation system of $(W,\varphi)$ on
$\rho(\Omega)Z$.
\end{lemma}
\begin{proof} Denote $(V|_{\rho(\Omega)Z}, \rho|_{\rho(\Omega)Z}, Q|_{\rho(\Omega)Z},
\rho(\Omega)T)$ by $(\widehat{V}, \hat{\rho}, \widehat{Q},
\widehat{T})$ for short, respectively. It is easy to verify that
$\hat{\rho}$ is a probability spectral measure on $\rho(\Omega)Z$.
For all $E\in\Sigma$, we have
\[\varphi(E)=Q\rho(E)T=Q\rho(\Omega\cap E \cap \Omega)T
=Q\rho(\Omega)\rho(E)\rho(\Omega)T=\widehat{Q}\rho(E)\widehat{T}.\]
For each $s\in S$,  we get $V_s\rho(\Omega)Z=\rho(s\Omega)V_s
Z=\rho(\Omega)V_s Z\subseteq\rho(\Omega)Z.$ So $\rho(\Omega)Z$ is a
invariant subspace of $V$. Then for all $s\in S$, we have
$\widehat{Q}\widehat{V}_s=QV_s|_{\rho(\Omega)Z}=W_sQ|_{\rho(\Omega)Z}
=W_s\widehat{Q}$ and
\[\widehat{V}_s\widehat{T}=V_s\rho(\Omega)T=\rho(s\Omega)V_s T
=\rho(\Omega)TW_s=\widehat{T} W_s.\] Thus $(\widehat{V}, \hat{\rho},
\widehat{Q}, \widehat{T})$ is a probability dilation system of
$(W,\varphi)$ on $\rho(\Omega)Z$.
%
%
\end{proof}

A key ingredient  in the dilation theory of operator valued measures developed in \cite{HLLL}  is the introduction of the
elementary dilation space $M_\varphi$ and the minimal dilation norm on the space $M_\varphi$. This is also needed  in the  the proof of Theorem 3.1.

Let $X$ be a Banach spaces and $(\Omega,\Sigma,\varphi,B(X))$ be
an operator-valued measure system. For any $E\in \Sigma$ and $x\in
X,$ define
\[\varphi_{x,E}:\Sigma\to X,\quad \varphi_{x,E}(F)=\varphi(E\cap F)x,
\quad \forall F\in \Sigma.\]
 Then it is easy to see that
$\varphi_{x,E}$ is a vector-valued measure on $(\Omega,\Sigma)$ of
$X$. Let $M_\varphi=\mbox{span}\{\varphi_{x,E}: x\in X, E\in \Sigma\}.$  Define $\|\cdot\|_\alpha: M_\varphi\to\R_+\cup\{0\}$ by
\[\left\|\sum^{N}_{i=1}c_i\varphi_{{x_i},{E_i}}\right\|_\alpha
=\sup_{E \in \Sigma}\left\|\sum_{i=1}^N c_i\varphi(E\cap
E_i)x_i\right\|_X\] for all $\sum^{N}_{i=1}c_i\varphi_{{x_i},{E_i}}\in
M_\varphi$.  Then $||\cdot ||_{\alpha}$ is a dilation norm on $M_{\varphi}$ (see \cite{HLLL}),
which is minimal in the sense that,
for any dilation norm $||\cdot ||_{\beta}$ on $M_{\varphi}$, there exists a
constant $C_\beta$ such that for any $\sum_{i=1}^N c_i \varphi_{{x_i},{E_i}}\in M_{\varphi},$
\[
\sup_{E\in\Sigma}\left\|\sum_{i=1}^N c_i\varphi(E\cap E_i)x_i\right\|_Y \leq C_\beta \left\|\sum_{i=1}^N c_i\varphi_{{x_i},{E_i}}\right\|_\beta,
\]
where $N>0,$ $\{c_i\}_{i=1}^{N}\subset\mathbb{C},$ $\{x_i\}_{i=1}^{N}\subset X$ and $\{E_i\}_{i=1}^{N}\subset \Sigma. $ Consequently we have
that
\[ \|f\|_\alpha\le C_\beta\|f\|_\beta, \qquad \forall f\in M_{\varphi}.\]

\begin{definition}\label{pr:449}
Let $(W,\varphi)$ be a projective isometric operator-valued system of
imprimitivity of $(S, \Sigma)$ on a Banach space $X$ with the
multiplier $\omega$. Assume that $\|\cdot\|_d$ is a norm on
$M_\varphi$ and denote the completion by $\mathcal{M}_\varphi$. Then
$\|\cdot\|_d$ is called a \emph{dilation norm} of $(W,\varphi)$ if:
\begin{enumerate}
\item The map $\rho:\Sigma\to B(\mathcal{M}_\varphi)$ defined by
$\rho(E)(\varphi_{x,F})=\varphi_{x,F\cap E}$ for all $x\in X$ and
$E,F\in\Sigma$ is an operator-valued measure;
\item The maps $T:X\to\mathcal{M}_\varphi$ and $Q:\mathcal{M}_\varphi\to X$ defined by
$T(x)=\varphi_{x,\Omega}$ and $Q(\varphi_{x,E})=\varphi(E)x$ for all
$x\in X$ and $E\in\Omega$ are bounded;
\item For all $s\in S$ the map $V_s$ on $\mathcal{M}_\varphi$ defined by
$V_s(\varphi_{x,E})=\varphi_{W_s x,s E}$ for all $x\in X$ and
$E\in\Sigma$ is an isometry.
\end{enumerate}
\end{definition}
\begin{lemma}\label{le:37} $(V, \rho, Q,T)$ is a probability dilation system of
$(W,\varphi)$.
\end{lemma}
\begin{proof} In the terminology of Definition \ref{pr:449},  it is easy to verify that $\rho$ is a probability
spectral measure of $\Sigma$ on $\mathcal{M}_\varphi$. Since
\[V_sV_t(\varphi_{x,E})=\varphi_{W_sW_t x,st E}=
\varphi_{\omega(s,t)W_{st} x,st E}=\omega(s,t)\varphi_{W_{st} x,st
E}=\omega(s,t)V_{st}(\varphi_{x,E}).\] for all $s,t\in S$, $x\in X$ and $E\in\Sigma$, we obtain
$V_sV_t=\omega(s,t)V_{st}$ for all $s,t\in S$. Thus $V$ is an isometric
$\omega$-representation of $S$ on $\mathcal{M}_\varphi$.

 For all
$s\in S$, $x\in X$ and $E,F\in\Sigma$, we have
\begin{eqnarray*}
  V_s\rho(F)(\varphi_{x,E}) &=& V_s(\varphi_{x,F\cap E})=\varphi_{W_s x, s(F\cap
E)}=\varphi_{W_s x, s F\cap sE} \\
    &=& \rho(sF)\varphi_{W_s x, sE}=
\rho(sF)V_s(\varphi_{x, E}).
\end{eqnarray*}
This implies that $V_s\rho(F)=\rho(sF)V_s$ for all $s\in S$ and
$F\in\Sigma$. Thus $(V,\rho)$ is a probability spectral
$\omega$-system of imprimitivity of $(S,\Sigma)$ on
$\mathcal{M}_\varphi$. From
\[Q\rho(E)T(x)=Q\rho(E)(\varphi_{x,\Omega})=Q(\varphi_{x,E})=\varphi(E)(x), \] we get  $\varphi(E)=Q\rho(E)T$ for all $E\in\Sigma$.

 For all $s\in S$,
$x\in X$ and $E\in\Sigma$, we have
\[QV_s(\varphi_{x,E})=Q(\varphi_{W_s x, sE})=\varphi(sE)W_s(x)=
W_s\varphi(E)x=W_sQ(\varphi_{x,E}).\] Thus $QV_s=W_sQ$ for all
$s\in S$.  Finally, for all $s\in S$ and $x\in X$ we have
\[V_sT(x)=V_s(\varphi_{x,\Omega})=\varphi_{W_s x, \Omega}=T(W_s x).\]
This implies that $V_s T=TW_s$ for all $s\in S$. Thus, $(V, \rho,
Q,T)$ is a probability dilation system of $(S, \Sigma)$ on
$\mathcal{M}_\varphi$.
\end{proof}

\begin{definition} The norm $||\cdot||_{\alpha}$ on the completion $\mathcal{M}_{\varphi}$ is called the {\it  minimal dilation norm}, and its induced probability dilation system $(V, \rho, Q,T)$ is called the {\it minimal dilation system }
$(W,\varphi)$.
\end{definition}

Now we show that every injective dilation system induces a natural dilation
norm. A dilation system $(V,\rho,Q,T)$ of $(W,\varphi)$ is said to be \textit{injective} if
$\sum \rho(E_i)T(x_i)=0$ whenever $\sum \varphi_{x_i, E_i}=0$ for all $x_i\in X$ and 
$E_i\in\Sigma$. This is equivalent to that the natural map 
from $M_\varphi$ to $\mathrm{span}\{\rho_{Tx,E}:x\in X, E\in\Sigma\}$ by 
$\varphi_{x,E}\mapsto \rho_{Tx,E}$ is injective (See Theorem 2.26 in \cite{HLLL}).

\begin{proposition}\label{pr:450}
Let $(W,\varphi)$ be a projective isometric operator-valued system of
imprimitivity of $(S,\Sigma)$ on $X$. If $(V,\rho,Q,T)$ is an injective
dilation system of $(W,\varphi)$ on $Z$, define $\|\cdot\|_d$ on
$M_\varphi$ by
\[\|\mu\|_d=\Big\|\sum_i\rho(E_i)T(x_i)\Big\|_Z\]
for all $\mu=\sum_i\varphi_{x_i,E_i}\in M_\varphi.$ Then
$\|\cdot\|_d$ is a dilation norm of $(W,\varphi)$. Assume that
$(V_d,\rho_d, Q_d,T_d)$ is the corresponding probability dilation
system of $(W,\varphi)$ on $\mathcal{M}_\varphi$. Thus, the natural
map $R$ from $\mathcal{M}_\varphi$ to $Z$ defined by
$$R(\mu)=\sum_i\rho(E_i)T(x_i)$$ for all $\mu=\sum_i\varphi_{x_i,E_i}\in
M_\varphi$ is a linear isometry and satisfies that for all $s\in S$
and $E\in\Sigma$ we have
\begin{enumerate}
  \item $R (V_d)_s=V_sR$;
  \item $R\rho_d(E)=\rho(E)R$;
  \item $Q_d=QR$;
  \item $RT_d=\rho(\Omega)T$.
\end{enumerate}
\end{proposition}
\begin{proof} By  Theorem 2.26 in \cite{HLLL},
we know that $\|\cdot\|_d$ is an norm, $Q_d$ and $T_d$ both are
well-defined, linear and bounded, and that $\rho_d$ is a probability
spectral measure. So we only need to prove that for all $s\in S$,
$(V_d)_s$ is an isometry. Let $\mu=\sum_i\varphi_{x_i,E_i}\in
M_\varphi$. Then we have
%
\begin{eqnarray*}
\|(V_d)_s(\mu)\|_d&=&\Big\|(V_d)_s\Big(\sum_i\varphi_{x_i,E_i}\Big)\Big\|_d
=\Big\|\sum_i \varphi_{W_sx_i,s E_i}\Big\|_d\\
&=&\Big\|\sum_i \rho(s E_i)T W_s(x_i)\Big\|
=\Big\|\sum_i \rho(s E_i)V_sT(x_i)\Big\|\nonumber\\
&=&\Big\|\sum_i V_s\rho(E_i)T(x_i)\Big\|
=\Big\|V_s\Big(\sum_i \rho(E_i)T(x_i)\Big)\Big\|\\
&=&\Big\|\sum_i \rho(E_i)T(x_i)\Big\| =\Big\|\sum_i
\varphi_{x_i,E_i}\Big\|_d\\&=&\|\mu\|_d.
\end{eqnarray*}
Thus, $\|\cdot\|_d$ is a dilation norm of $(W,\varphi)$. It
is clear by definition that $R$ is an isometric embedding from
$\mathcal{M}_\varphi$ into $Z$. For all $s\in S$, $x\in X$ and
$F\in\Sigma$, we have
\begin{eqnarray*}
  R(V_d)_s(\varphi_{x,F}) &=& R(\varphi_{W_sx,sF})=\rho(sF)TW_s(x)=\rho(sF)V_sT(x) \\
    &=& V_s\rho(F)T(x)=V_sR(\varphi_{x,F}).
\end{eqnarray*}
Thus  $R(V_d)_s=V_sR$ for all $s\in S$. If $x\in X$ and
$E,F\in\Sigma$, then  we have
\begin{eqnarray*}
  R\rho_d(E)(\varphi_{x,F}) &=& R(\varphi_{x,E\cap F})=\rho(E\cap F)T(x) \\
    &=& \rho(E)\rho(F)T(x)
=\rho(E)R(\varphi_{x,F}),
\end{eqnarray*}
and thus $R\rho_d(E)=\rho(E)R$ for all $E\in\Sigma$. Finally, if
$x\in X$ and $F\in\Sigma$, then we get
\[RT_d(x)=R(\varphi_{x,\Omega})=\rho(\Omega)T(x),\]
\[Q_d(\varphi_{x,F})=\varphi{F}(x)=Q\rho(F)T(x)
=QR(\varphi_{x,F}).\] Therefore we get that $RT_d=\rho(\Omega)T$ and
$Q_d=QR$.
\end{proof}

%

Now we are ready to complete the proof of Theorem \ref{main-thm1}.

\bigskip

\noindent{\bf Proof of Theorem \ref{main-thm1}:} Let $(W,\varphi)$ be a projective isometric operator-valued system of imprimitivity
of $(S, \Sigma)$ on a Banach space $X$, and  $(\rho,Q,T)$ be the minimal dilation
system of $\varphi$ on $\mathcal{M}_\varphi$. For any $s\in S$, define
$V_s$ on $\mathcal{M}_\varphi$ by
$V_s(\varphi_{x,E})=\varphi_{W_s x, s E}$
for all $x\in X$ and $E\in\Sigma$. Then for all
$\mu=\sum_j\varphi_{x_j,E_j}\in M_\varphi$
\begin{eqnarray}\label{eq:m1}
\|V_s(\mu)\|&=&
  \left\|V_s(\sum \varphi_{x_j,E_j})\right\| =
  \left\|\sum\varphi_{W_s x_j, s E_j}\right\| \nonumber\\
   &=& \sup_{F\in\Sigma}\left\|\sum\varphi(sE_j\cap F)W_s x_j\right\|
   = \sup_{F\in\Sigma}\left\|\sum W_s\varphi(E_j\cap s^{-1}F) x_j\right\| \nonumber\\
   &=& \sup_{F\in\Sigma}\left\|W_s\left(\sum\varphi(E_j\cap F) x_j\right)\right\|
   \leq \sup_{F\in\Sigma}\left\|\sum\varphi(E_j\cap F) x_j\right\| \nonumber \\
   &=& \left\|\sum \varphi_{x_j,E_j}\right\|
   =\|\mu\|.
\end{eqnarray}
For all $s,t\in S$, $x\in X$ and $E\in\Sigma$, we have
\[  V_s V_t(\varphi_{x,E}) =
  V_s(\varphi_{W_t x, t E})=\varphi_{W_s W_t x, s(tE)}
   = \varphi_{\omega(s,t)W_{st} x, (st) E}=
   \omega(s,t)V_{st}(\varphi_{x,E}).
\]
This implies  that $V_s V_t=\omega(s,t)V_{st}$ for each $s,t\in S$.
Thus, $V$ is an isometric $\omega$-representation of $S$ on
$\mathcal{M}_\varphi$. Furthermore,  since for all $s\in S$, $x\in X$ and
$E,F\in\Sigma$ we have
\begin{eqnarray*}
  V_s\rho(E)(\varphi_{x,F}) &=& V_s(\varphi_{x,E\cap F})
=\varphi_{W_s x, s(E\cap F)}=\varphi_{W_s x, sE\cap sF} \\
   &=& \rho(E)(\varphi_{W_sx,sF})=\rho(sE)V_s(\varphi_{x,F}),
\end{eqnarray*}
we obtain that $V_s\rho(E)=\rho(sE)V_s$ for all $s\in S$ and
$E\in\Sigma$. Therefore $\|\cdot\|_{\mathcal{M}_\varphi}$ is a dilation
norm of $(W,\varphi)$, and so,  by Lemma \ref{le:37}, $(V,\rho,Q,T)$ is a
probability dilation system of $(W,\varphi)$.\qed
\bigskip

While the dilated probability projective spectral system of imprimitivity in general is based on a Banach space, it is natural to ask whether the dilated space can also be taken as a Hilbert space if the considered isometric  operator-valued system of
imprimitivity is based on a Hilbert space.  Since there exists an example of a Hilbert space based operator-valued measure which has a Banach space dilation but not a Hilbert space dilation (See Theorem E and the subsequent discussion in the introduction to \cite{HLLL}),  it seems that this Banach space restriction is necessary. However, as with Naimark's dilation theorem (c.f. \cite{Aev, HA, Pa}) , if the condition of positivity is imposed on the isometric operator-valued system a Hilbert space dilation is possible.

\begin{theorem}\label{th:H67}
 Let $\mathcal {G}$ be a locally compact Hausdorff topological group
acting on a measurable space $(\Omega,\Sigma)$.
Let $(U, \varpi)$ be an isometric positive operator-valued system of
imprimitivity of $(\mathcal {G}, \Sigma)$ on a Hilbert
space $\cH$. Then there is an isometric (orthogonal) projection-valued system of
imprimitivity $(\widetilde{U},\pi)$ of $(\mathcal {G}, \Sigma)$ on a
Hilbert space $\cK$ and a bounded linear operator
$V:\cH\rightarrow \cK$ such that
\[ \widetilde{U}_g V=V\, U_g \ \ \mbox{ and } \ \ \varpi(E)=V^* \pi(E) V  \]
for all $g\in \mathcal {G}$ and\, $E\in \Sigma$.
\end{theorem}
\begin{proof}
Let $\cM$ be the linear space of all vector measures from $\Sigma$
to $\cH$. Define
$$\cM_\varpi=\mathrm{span}\{\varpi_{x,E}:x\in\cH,E\in\Sigma\}$$
which is a linear subspace of $\cM$ induced by $\varpi$. Now we
define a sesquilinear functional $\langle \,, \rangle$ on
$\cM_\varpi$ by
\begin{equation}
\langle M_1,M_2\rangle=\sum_{i=1}^n\sum_{j=1}^m\alpha_i\bar{\beta}_j
\langle\varpi(E_i\cap F_j)\,x_i, y_j\rangle_\cH
\end{equation}
for each $M_1=\sum_{i=1}^n \alpha_i \varpi_{x_i,E_i}$ and
$M_2=\sum_{j=1}^m \beta_j \varpi_{y_j,F_j}$ in $\cM_\varpi$. Since
\[ \sum_{i=1}^n\sum_{j=1}^m\alpha_i\bar{\beta}_j
\langle\varpi(E_i\cap F_j)\,x_i, y_j\rangle_\cH=
\sum_{j=1}^m\bar{\beta}_j \langle M_1(F_j), y_j\rangle_\cH=
\sum_{i=1}^n\alpha_i \langle\,x_i, M_2(E_i)\rangle_\cH,\] we get that the
sesquilinear functional $\langle \,, \rangle$ is well-defined. For
any $M=\sum_{i=1}^n \alpha_i \varpi_{x_i,E_i}$ (without losing the
generality, we can assume that $E_i$'s are disjoint from each
other), we have that
\begin{equation*}
\langle M,M \rangle=\sum_{i=1}^n\sum_{j=1}^n\alpha_i\bar{\alpha}_j
\langle \varphi(E_i\cap E_j)x_i,x_j\rangle_\cH=\sum_{i=1}^n
|\alpha_i|^2 \langle \varphi(E_i)x_i,x_i\rangle_\cH\ge 0.
\end{equation*}
Thus  $\langle M,M \rangle=0$ if and only if $|\alpha_i|^2 \langle
\varphi(E_i)x_i,x_i\rangle_\cH=0$ for each $1\le i\le n$. For any
$E\in\Sigma$ and $1\le i\le n$,
\begin{eqnarray*}
\|\alpha_i\varphi^{1/2}(E\cap E_i)x_i\|^2&=&|\alpha_i|^2
\langle \varphi^{1/2}(E\cap E_i)x_i,\varphi^{1/2}(E\cap E_i)x_i\rangle_\cH\\
&=&|\alpha_i|^2 \langle \varphi(E\cap E_i)x_i,x_i\rangle_\cH\\&\le&
|\alpha_i|^2 \langle \varphi(E_i)x_i,x_i\rangle_\cH=0.
\end{eqnarray*}
Thus, for all $E\in\Sigma$, we have
$$M(E)=\sum_{i=1}^n \alpha_i\varpi(E\cap E_i)x_i
=\sum_{i=1}^n \varpi^{1/2}(E\cap E_i)(\alpha_i\varphi^{1/2}(E\cap
E_i)x_i)=0,$$ which implies that $M=0.$ Thus
$\langle\,,\rangle$ is positive definite and hence  an inner product
on $\cM_\varpi$. Let $\cK$ be the completion of the inner product
space $\cM_\varpi$.

We define a linear map $V:\cH\to \cK$ by
\begin{equation}
    V(x)=\varpi_{x,\Omega} \quad \mbox{ for all } \, x\in\cH.
\end{equation}
Then
\begin{eqnarray*}
\|V(x)\|^2&=&\langle V(x),V(x)\rangle=
\langle\varpi_{x,\Omega},\varpi_{x,\Omega} \rangle =\langle
\varpi(\Omega)x,x\rangle\\&=&\langle\varpi^{1/2}(\Omega)x,
\varpi^{1/2}(\Omega)x\rangle=\|\varpi^{1/2}(\Omega)x\|^2
\end{eqnarray*}
for all $x\in\cH$. So $V$ is bounded with
$$\|V\|=\|\varpi^{1/2}(\Omega)\| =\|\varpi(\Omega)\|^{1/2}.$$

For any $g\in\cG$, define a linear map
$\widetilde{U}_g:\cM_\varpi\to\cM_\varpi$ by
\begin{equation}
    \widetilde{U}_g(M)(E)=U_g M(g^{-1}E)
\end{equation}
for all $M\in\cM_\varpi$ and $E\in\Sigma$. Let
$M=\sum_{i=1}^n\alpha_i\varpi_{x_i,E_i}$ in $\cM_\varpi$. Then for
all $E\in\Sigma$, we get
\begin{eqnarray*}
  \widetilde{U}_g(M)(E) &=& U_g M(g^{-1}E)\\
   &=& \sum_{i=1}^n\alpha_i U_g \varpi(g^{-1}E \cap E_i)x_i\\
   &=& \sum_{i=1}^n\alpha_i \varpi(E \cap g\cdot E_i)U_g x_i
\end{eqnarray*}
That is,
\begin{equation}
\widetilde{U}_g(M)=\sum_{i=1}^n\alpha_i\varpi_{U_g(x_i),\,g\cdot
E_i} \in\cM_\varpi.
\end{equation}
Thus, it is clear that $\widetilde{U}_g$ is well-defined and
surjective. For any $M_1,M_2\in \cM_\varpi$ with
$M_2=\sum_{j=1}^m\beta_j\varpi_{y_j,F_j}$, we obtain that
\begin{eqnarray*}
  \langle \widetilde{U}_g(M_1), \widetilde{U}_g(M_2)\rangle &=&
  \Big\langle \widetilde{U}_g(M_1), \sum_{j=1}^m\beta_j\varpi_{U_g(y_j),\,g\cdot F_j}\Big\rangle \\
   &=& \sum_{j=1}^m\bar{\beta}_j\langle\widetilde{U}_g(M_1)(g\cdot F_j), U_g y_j\rangle \\
   &=& \sum_{j=1}^m\bar{\beta}_j\langle U_g M_1(F_j), U_g y_j\rangle \\
   &=& \sum_{j=1}^m\bar{\beta}_j\langle M_1(F_j), y_j\rangle\\
   &=& \Big\langle M_1,\sum_{j=1}^m\beta_j\varpi_{y_j,F_j}\Big\rangle \\
   &=& \langle M_1,M_2\rangle.
\end{eqnarray*}
So $\widetilde{U}_g$ can be uniquely extended to be a unitary
operator on $\cK$ for each $g\in\cG$. For any $g_1,g_2\in\cG$, $M\in
\cM_\varpi$ and $E\in\Sigma$, we have
\begin{eqnarray*}
  \widetilde{U}_{g_1}\widetilde{U}_{g_2}(M)(E) &=&
  U_{g_1}\widetilde{U}_{g_2}(M)(g_1^{-1}\cdot E)
   = U_{g_1}U_{g_2}M(g_2^{-1}g_1^{-1}\cdot E) \\
   &=& U_{g_1g_2}M((g_1g_2)^{-1}\cdot E)
   = \widetilde{U}_{g_1g_2}(M)(E).
\end{eqnarray*}
This implies that
$\widetilde{U}_{g_1}\widetilde{U}_{g_2}=\widetilde{U}_{g_1g_2}$ and
the map $\widetilde{U}:\mathcal {G}\to B(\cK)$ is a unitary
representation of $\cG$ on a Hilbert space $\cK$.

Now we define a representation $\widetilde{U}:\mathcal {G}\to
B(\cK)$ by
$$\widetilde{U}(g)\left(\sum_{i=1}^n\varpi_{{x_i},{E_i}}\right)=\sum_{i=1}^n
  \varpi_{U_g{x_i},g\cdot E_i}.$$
Since $U$ is a strongly-continuous unitary representation, it is
easy to see that $\widetilde{U}$ is also a strongly-continuous
unitary representation. Now we show that
\[ \widetilde{U}_g V=V U_g, \quad\mbox{and} \quad \widetilde{U}_g \pi(E)\widetilde{U}_{g^{-1}}
=\pi(g\cdot E).\] Let $g\in G .$ The we have
$g\cdot\Omega=g\cdot\Omega\cap\Omega=g\cdot(\Omega\cap
g^{-1}\cdot\Omega)=gg^{-1}\Omega=\Omega.$

Since, for any $x\in\cH,$
\[\widetilde{U}_g V(x)=\widetilde{U}_g(\varpi_{x,\Omega})
=\varpi_{U_g(x),g\cdot \Omega}=\varpi_{U_g(x),\Omega}=V U_g(x),\]
we obtain that $\widetilde{U}_g V=V U_g.$

A bounded linear operator $V:\cH\to \cK$, and an orthogonal
projection-valued measure $\pi:\Sigma\to B(\cK)$ such that
$$\varpi(E)=V^* \pi(E) V.$$
The set $M_\varpi=\mbox{span}\{\varpi_{x,E}: x\in \cH, E\in
\Sigma\}$ is dense in $\cK$.
 For every $E\in\Sigma$, the linear map
 $\pi(E):\widetilde{M}_\varpi\to\widetilde{M}_\varpi$ is
 $$\pi(E)\left(\sum_{i=1}^n\varpi_{{x_i},{E_i}}\right)=\sum_{i=1}^n
  \varpi_{{x_i},{E\cap E_i}}.$$

Since $\widetilde{U}$ is an unitary representation, we only need to
prove $\widetilde{U}_g\pi(E)=\pi(g\cdot E)\widetilde{U}_g.$ For any
$\sum_{i=1}^n\varpi_{{x_i},{E_i}}\in\cK,$
\[\widetilde{U}_g\pi(E)\left(\sum_{i=1}^n\varpi_{{x_i},{E_i}}\right)
=\widetilde{U}_g\left(\sum_{i=1}^n\varpi_{{x_i},{E_i\cap
E}}\right)=\sum_{i=1}^n\varpi_{U_g(x_i),{g\cdot(E_i\cap E)}},\]
\[\pi(g\cdot E)\widetilde{U}_g\left(\sum_{i=1}^n\varpi_{{x_i},{E_i}}\right)
=\pi(g\cdot E)\left(\sum_{i=1}^n\varpi_{U_g(x_i),{g\cdot
E_i}}\right)=\sum_{i=1}^n\varpi_{U_g(x_i),{g\cdot (E_i\cap E)}}.\]
Hence $\widetilde{U}_g \pi(E)\widetilde{U}_{g^{-1}} =\pi(g\cdot E).$
\end{proof}


\section{Dilation of isometric representation induced framings}

In this section we investigate the dilation property for framings that are induced by a projective isometric representation on a discrete group and based on a Banach space.

Recall that if $\{x_j\}_{j\in\J}$ is a frame for a Hilbert space $H$
with the frame transform $S=V^*V$, then we get a reconstruction
operator $S^{-1}$ satisfying
\[x=\sum_{j\in\J}\langle x, S^{-1}x_j \rangle x_j\]
and this series converges unconditionally for all $x\in H$. In
\cite{CHL}, Casazza, Han and Larson introduce a natural definition for a
framing in a Banach space from this representation.
\begin{definition} (i) A \emph{framing} for a Banach space $X$ is a pair of sequences
$\{x_j,f_j\}_{j\in\J}$ with $x_j \in X$ and $f_j \in X^*$ for all
$j\in\J$ such that
\[x=\sum_{j\in\J}\langle x, f_j \rangle x_j \quad \mbox{ for all } x\in X\]
and this series converges unconditionally for all $x\in H$.

(ii) Let $X$ be a Banach space and $G$ a discrete countable group
with unit $ u $. Assume that $\theta$ is a projective
isometric representation of $G$ on $X$ with a multiplier $m$. A $\theta$-induced framing for $X$ is framing of the form  $\{\theta_g
x_j,\theta_{g^{-1}}^* x^*_j\}_{g\in G, j\in\J}$, where
and
that there exist $\{x_j\}_{j\in\J}\subset X$ and
$\{x^*_j\}_{j\in\J}\subset X^*$. When $\J$ is a singleton we say that this is a single-window framing generated by $\theta$, otherwise it is called a multi-window framing generated by $\theta$.
\end{definition}

A  $\theta$-induced framing naturally induces an operator-valued system of imprimitivity: Let $\Sigma$ be the $\sigma$-algebra consiting of all the subsets of $G$, and define
$$
\varphi(E) = \sum_{j\in J, g\in E}\theta_g
x_j\otimes\theta_{g^{-1}}^* x^*_j.
$$
Then $(\varphi, \theta)$ is a projective isometric operator-valued system of imprimitivity, and consequently by our main dilation theorem (Theorem 3.2) it can be
dilated to a probability projective isometrc spectral system of imprimitivity. However, the framing nature of the dilated system becomes somewhat lost in this dilation. Our main result (Theorem 4.2) of this section shows that  we can dilate $(\varphi, \theta)$ to  a  probability projective isometrc spectral system of imprimitivity that is also induced by  a framing on the dilated Banach space. We accomplish this by directly dilating a $\theta$-induced framing to a projective representation induced unconditional basis. The proof uses techniques from both the proof of our main dilation theorem and
the proof of the dilation theory of group representation induced dual frame pairs in Hilbert spaces (cf. \cite{Han2}). It is a  dilation theorem for (finite or infinite) multi-window framings generated by projective isometric representations of discrete groups on Banach spaces.

\begin{theorem}\label{th:410} Let $X$ be a Banach space and $G$ a discrete countable group
with unit $u$. Assume that $\theta$ is a projective
isometric representation of $G$ on $X$ with a multiplier $m$ and
that there exist $\{x_j\}_{j\in\J}\subset X$ and
$\{x^*_j\}_{j\in\J}\subset X^*$ such that $\{\theta_g
x_j,\theta_{g^{-1}}^* x^*_j\}_{g\in G, j\in\J}$ is a framing for
$X$. Then there exists a Banach space $Z$ with an unconditional
basis $\{e_{g,j},e_{g,j}^*\}_{g\in G,j\in\J}$ such that the map
$T:X\to Z$ defined by
\[T(x)=\sum_{g\in G}\sum_{j\in\J}\langle x, \theta_{g^{-1}}^* x^*_j \rangle e_{g,j}\]
is an into isomorphism,  and the map $S:Z\to X$ given by
\[S\Big(\sum_{g\in G}\sum_{j\in\J} a_{g,j} e_{g,j}\Big)=\sum_{g\in G}
\sum_{j\in\J} a_{g,j} \theta_g x_j\] is contractive and surjective,
the map $\lambda:G\to B(Z)$ defined by
\[\lambda_h\Big(\sum_{g\in G} \sum_{j\in\J} a_{g,j} e_{g,j}\Big)=
\sum_{g\in G} \sum_{j\in\J} m(h,g) a_{g,j} e_{hg,j}\] is a
projective isometric representation of $G$ on $Z$ with the same
multiplier $m$ and
\[\|\lambda_h\|\le\|\theta_h\| \quad \mbox{ for all } h\in G. \]
Furthermore, in this case, we have
\begin{enumerate}
  \item[(i)] $S T=I_X$ and $T^*(e_{g,j}^*)=\theta_{g^{-1}}^* x^*_j$ for all $g\in G$ and $j\in\J$;
  \item[(ii)] For each $h\in G$, we have
  \[\lambda_h^*\Big(\sum_{g\in G} \sum_{j\in\J} b_{g,j} e_{g,j}^*\Big)=
  \sum_{g\in G} \sum_{j\in\J} m(h^{-1},g) b_{g,j} e_{h^{-1}g,j}^*.\]
  Thus, $e_{g,j}=\lambda_g e_{ u ,j}$ and $e_{g,j}^*=\lambda_{g^{-1}}^*
  e_{ u ,j}^*$ for each $g\in G$ and $j\in\J$;
  \item[(iii)] For all $g\in G$, we have
  \[\theta_g S=S\lambda_g \ \ \mbox{ and } \ \ \lambda_g T=T \theta_g.\]
\end{enumerate}
\end{theorem}
\begin{proof}
First we need to construct our dilation Banach space $Z$ with an
unconditional basis $\{e_{g,j}, e^*_{g,j}\}_{g\in G,j\in\J}$. To do
this, define a new linear space $F(G\times \J)$ by
\[F(G\times\J)=\{f:G\times\J\to\C:f \mbox{ is a function with finite support} \}\]
with the norm on $F(G\times\J)$ given by
\[\|f\|=\sup_{E\subset G\times\J}\Big\|\sum_{(g,j)\in E} f(g,j)\theta_g x_j\Big\|
=\max_{E\subset G\times\J}\Big\|\sum_{(g,j)\in E} f(g,j)\theta_g
x_j\Big\|.\] Let $\{e_{g,j}\}_{g\in G,j\in\J}$ be the natural unit
vectors in $F(G\times\J)$ and let $Z$ denote the completion of
$F(G\times\J)$ under the above norm $\|\cdot\|$. Note that
$\|e_{g,j}\|=\|\theta_g x_j\|>0$ for each $g\in G$ and $j\in\J$.
Thus, it is easy to see that $\{e_{g,j}\}_{g\in G,j\in\J}$ is a
contractive unconditional basis of $Z$. Now, we define a map $T:X\to
Z$ by
\[T(x)=\sum_{g\in G}\sum_{j\in\J}\langle x, \theta_{g^{-1}}^* x^*_j \rangle e_{g,j}
\quad \mbox{ for all } x\in X.\] It is basic to prove that $T$ is
well-defined, which we leave to interested readers. Since
\[C=\sup_{\|x\|\le 1}\sup_{E\subset G\times\J}\Big\|
\sum_{(g,j)\in E}\langle x, \theta_{g^{-1}}^* x^*_j \rangle \theta_g
x_j \Big\|<\infty.\] Then for all $x\in X$,
\[\|T(x)\|=\sup_{E\subset G\times\J}\Big\|
\sum_{(g,j)\in E}\langle x, \theta_{g^{-1}}^* x^*_j \rangle \theta_g
x_j \Big\| \leq C\|x\|.\] It follows that $T$ is bounded with
$\|T\|\le C$. Also, define a map $S:F(G\times\J)\to X$ by
\[S( f )=S\Big(\sum_{g\in G}\sum_{j\in\J} f(g,j) e_{g,j}\Big)=\sum_{g\in G}
\sum_{j\in\J} f(g,j) \theta_g x_j \quad \mbox{ for all } f\in
F(G\times\J).\] That is, $S(e_{g,j})=\theta_g x_j$ for all $g\in G$
and $j\in\J$. Then for any $f\in F(G\times\J)$,
\[\|S(f)\|=\Big\|\sum_{g\in G}
\sum_{j\in\J} f(g,j) \theta_g x_j\Big\| \le\sup_{E\subset
G\times\J}\Big\|\sum_{(g,j)\in E} f(g,j) \theta_g x_j\Big\|
=\|f\|.\] Thus, $S$ is contractive and can be uniquely extended to
$Z$. For any $h\in G$, the map $\lambda_h$ on $F(G\times\J)$ defined
by
\[\lambda_h(f)=\lambda_h\Big(\sum_{g\in G}\sum_{j\in\J} f(g,j) e_{g,j}\Big)
=\sum_{g\in G}\sum_{j\in\J} m(h,g) f(g,j) e_{hg,j} \ \ \mbox{ for
all } f\in F(G\times\J).\] That is, $\lambda_h(e_{g,j})=m(h,g)
e_{hg,j}$ for all $g\in G$ and $j\in\J$. Moreover, we obtain that
\begin{eqnarray*}\label{eq:42}
  \|\lambda_h(f)\| &=& \Big\|\sum_{g\in G}\sum_{j\in\J} m(h,g) f(g,j) e_{hg,j}\Big\| \nonumber \\
  &=&\Big\|\sum_{g\in G}\sum_{j\in\J} m(h,h^{-1}g) f(h^{-1}g,j) e_{g,j}\Big\| \nonumber \\
   &=& \sup_{E\subset G\times\J} \Big\|\sum_{(g,j)\in E} m(h,h^{-1}g) f(h^{-1}g,j) \theta_{g} x_j\Big\| \nonumber \\
    &=& \sup_{E\subset G\times\J} \Big\|\sum_{(hg,j)\in E} m(h,g) f(g,j) \theta_{hg} x_j\Big\| \nonumber\\
   &=& \sup_{E\subset G\times\J} \Big\|\sum_{(g,j)\in E} f(g,j) \theta_{h}\theta_{g} x_j\Big\| \nonumber\\
   &=& \sup_{E\subset G\times\J} \Big\|\theta_{h}\Big(\sum_{(g,j)\in E} f(g,j)\theta_{g} x_j\Big)\Big\|\\
   &=&\sup_{E\subset G\times\J} \Big\|\sum_{(g,j)\in E} f(g,j)\theta_{g} x_j\Big\| \nonumber\\
   &=&\|f\|. \nonumber
\end{eqnarray*}
So $\lambda_h$ is isometric and can be uniquely extended to $Z$.
Clearly, for any $h_1,h_2,g\in G$,
\begin{eqnarray*}\lambda_{h_1}\lambda_{h_2}(e_g)&=&m(h_2,g)\lambda_{h_1}(e_{h_2g})
=m(h_1,h_2g)m(h_2,g)e_{h_1h_2g}\\&=&
m(h_1,h_2)m(h_1h_2,g)e_{h_1h_2g}=
m(h_1,h_2)\lambda_{h_1h_2}(e_g).\end{eqnarray*} It implies that
\[\lambda_{h_1}\lambda_{h_2}=m(h_1,h_2)\lambda_{h_1h_2} \ \ \mbox{
for each } h_1,h_2\in G.\] Thus, $\lambda$ is a projective isometric
representation of $G$ on $Z$ with the same multiplier $m$. Since for
all $x\in X$, we have
\[ST(x)=S\Big( \sum_{g\in G}\sum_{j\in\J} \langle x, \theta_{g^{-1}}^* x^*_j \rangle e_{g,j} \Big)
=\sum_{g\in G} \sum_{j\in\J} \langle x, \theta_{g^{-1}}^* x^*_j
\rangle \theta_{g} x_j=x.\] That is, $ST=I_X,$ which implies that
$T$ is an into isomorphism and that $S$ is surjective. Let
$\{e_{g,j}^*\}_{g\in G,j\in\J}$ be the biorthogonal functionals of
$\{e_{g,j}\}_{g\in G,j\in\J}$. For each $x\in X$, $h\in G$ and
$j\in\J$,
\[\langle x,T^*(e_{h,j}^*)\rangle=\langle T(x),e_{h,j}^*\rangle=
\Big\langle \sum_{g\in G}\sum_{i\in\J} \langle
x,\theta^*_{g^{-1}}x^*_i\rangle e_{g,i},
 e_{h,j}^*\Big\rangle
=\langle x,\theta^*_{h^{-1}} x^*_j \rangle.\] Hence
$T^*(e_{h,j}^*)=\theta^*_{h^{-1}}x^*_j$ for all $h\in G$ and
$j\in\J$. For every $s,h,g\in G$ and $i,j\in\J$,
\begin{eqnarray*}\langle e_{s,i}, \lambda_h^*(e_{g,j}^*) \rangle&=&\langle \lambda_h(e_{s,i}),e_{g,j}^* \rangle
\\&=&m(h,s)\langle e_{hs,i},e_{g,j}^* \rangle\\&=&
m(h,h^{-1}g)\langle e_{s,i},e_{h^{-1}g,j}^* \rangle\\&=&
\langle e_{s,i}, \overline{m(h,h^{-1}g)}e_{h^{-1}g,j}^* \rangle\\
&=&\langle e_{s,i}, m(h^{-1},g)e_{h^{-1}g,j}^* \rangle.
\end{eqnarray*}
Then $\lambda_h(e_{g,j}^*)=m(h^{-1},g)e_{h^{-1}g,j}^*$. That is,
\[\lambda_h^*\Big(\sum_{g\in G}\sum_{j\in\J} b_{g,j} e_{g,j}^*\Big)
=\sum_{g\in G}\sum_{j\in\J} m(h^{-1},g) b_{g,j} e_{h^{-1}g,j}^*=
\sum_{g\in G}\sum_{j\in\J} m(h^{-1},hg) b_{hg,j} e_{g,j}^*.\] Thus,
\[e_{g,j}=m(g, u )e_{g u ,j}=\lambda_g
e_{ u ,j} \ \ \mbox{ and } \ \
e_{g,j}^*=m(g, u )e^*_{g u ,j}=\lambda^*_{g^{-1}}e^*_{ u ,j}\]
for all $g\in G$ and $j\in\J$. Furthermore, we have
\[\theta_h S(e_{g,j})=\theta_h\theta_g x_j=m(h,g)\theta_{hg} x_j
=S(m(h,g)e_{hg,j})=S\lambda_h(e_{g,j}),\] and
\begin{eqnarray*}
  \lambda_h T(x) &=& \lambda_h\Big( \sum_{g\in G}\sum_{j\in\J}
  \langle x, \theta^*_{g^{-1}} x^*_j\rangle e_{g,j}\Big)\\
  &=& \sum_{g\in G}\sum_{j\in\J} m(h,g)\langle x, \theta^*_{g^{-1}} x^*_j\rangle e_{hg,j}\\
  &=& \sum_{g\in G}\sum_{j\in\J} \langle x, \overline{m(h,h^{-1}g)}\theta^*_{({h^{-1}g)}^{-1}} x^*_j\rangle e_{g,j}\\
   &=& \sum_{g\in G}\sum_{j\in\J} \langle x, m(h^{-1},g)\theta^*_{g^{-1}h} x^*_j\rangle e_{g,j}\\
  &=& \sum_{g\in G}\sum_{j\in\J} \langle x, \theta^*_h\theta^*_{g^{-1}} x^*_j\rangle e_{g,j} \\
  &=& \sum_{g\in G}\sum_{j\in\J} \langle \theta_h x, \theta^*_{g^{-1}} x^*_j\rangle e_{g,j}\\
  &=& T\theta_h(x).
\end{eqnarray*}
Thus, for all $h\in G$, we obtain that
\[\theta_h S=S\lambda_h \ \ \mbox{ and } \ \ \lambda_h T= T\theta_h.\]
Then we complete the proof.
\end{proof}

Framings generated by projective isometric representations of discrete groups on Banach spaces appear in various contexts of group representations theory, Gabor/wavelet representations  for function spaces, and $p$-frames for shift-invariant subspaces. For the purpose of demonstration  we include two simple examples at the end of this paper.

\begin{example} Let $X$ be a Banach space and $G$ be a discrete countable group.
Let $I$ be an isometric representation of $G$ on $X$. Assume that
there is $x\in X$ and $f\in X^*$ such that $\{I_g x, I_{g^{-1}}^*
f\}_{g\in G}$ is a framing for $X$. Let $\varphi$ be the induced OVM
defined by
\[\varphi(E)=\sum_{g\in E} I_g x \otimes I_{g^{-1}}^* f\]
for all $E\subset G$. Then $(I,\varphi)$ is an operator-valued isometric system
of imprimitivity. Actually, for any
$g\in G$, $E\subset G$ and $x\in X$, we have
\begin{eqnarray*}
  I_g\varphi(E)I_{g^{-1}}(x) &=&
  I_g\sum_{h\in E} \langle I_{g^{-1}}(x), I_{h^{-1}}^* f \rangle I_h x \\
   &=& \sum_{h\in E} \langle x, I_{g^{-1}}^*I_{h^{-1}}^* f \rangle I_g I_h x \\
   &=& \sum_{h\in E} \langle x, I_{(gh)^{-1}}^* f \rangle I_{gh} x  \\
   &=& \sum_{h\in g E} \langle x, I_{h^{-1}}^* f \rangle I_h x  \\
   &=& \varphi(g E)(x).
\end{eqnarray*}
That is, $I_g\varphi(E)I_{g^{-1}}=\varphi(g E)$ for all $g\in G$ and
$E\subset G$. Thus, $(I,\varphi)$ is an operator-valued isometric system
of imprimitivity. If $\{I_g x, I_{g^{-1}}^* f\}_{g\in G}$ is an unconditional basis of $X$, then
$(I,\varphi)$ is an isometric spectral system of imprimitivity.
\end{example}
\begin{example} Let $1\le p \le \infty$, $\Gamma$ be a countable
set, and let $X$ be a normed linear space and $X^*$ be its dual. We
say that $\{g_\lambda:\lambda\in \Gamma\}\subset X^*$ is a $p$-frame
for $X$ if the map $T$ defined by
\[T:X\ni f\mapsto \{\langle f, g_\lambda \rangle\}_{\lambda\in\Gamma}\in\ell^p(\Gamma),\]
is both bounded and bounded below, i.e., there exists a positive
constant $C$ such that for all $f\in X$
\[C^{-1}\|f\|_X\le \Big(\sum_{\lambda\in\Gamma}|\langle f, g_\lambda \rangle|^p\Big)^{1/p}
\le C\|f\|_X\] for $1\le p<\infty$, and for all $f\in X$
\[C^{-1}\|f\|_X\le \sup_{\lambda\in\Gamma}|\langle f, g_\lambda \rangle|\le C\|f\|_X\]
for $p=\infty$. In [AST], Aldroubi, Sun and Tang derived necessary
and sufficient conditions for an indexed family
$\{\phi_i(\cdot-j):1\le i\le r, j\in\Z^d\}$ to construct a $p$-frame
for the shift invariant space
\[V_p(\Phi)=\left\{\sum_{i=1}^r\sum_{j\in \Z^d}d_i(j)\phi_i(\cdot-j):(
d_i(j))_{j\in\Z^d}\in\ell^p\right\}, \ \ 1\le p\le \infty.\] They
prove that if $\{\phi_i(\cdot-j):1\le i\le r, j\in\Z^d\}$ is a
$p$-frame for $V_p(\Phi)$, then there exists $\{\psi_1,...,\psi_r\}$
such that for all $f\in V_p(\Phi)$
\[f=\sum_{i=1}^r\sum_{j\in\Z^d}\langle f, \psi_i(\cdot-j)\rangle\phi_i(\cdot-j)
=\sum_{i=1}^r\sum_{j\in\Z^d}\langle f,
\phi_i(\cdot-j)\rangle\psi_i(\cdot-j).\] This is a multi-window
framing generated by the isometric representation of the additive
group $\Z^d$ on a shift invariant subspace of $L^p(\R^d)$.
\end{example}


\begin{thebibliography}{1}

\bibitem{Aev}W. Arveson, \emph{Dilation theory yesterday and today,} Operator Theory: Advances and Applications, vol. 207 (2010), 99--123.
%

\bibitem{AST} A. Aldroubi, Q. Sun and W. Tang, {\em $p$-frames and
shift-invariant subspaces of $L^{p}$,} J.  Fourier Anal. Appl., {\bf 7} (2001), 1--21

\bibitem{CHL} P. G. Casazza, D. Han, and D. R. Larson,
 {\em Frames for Banach spaces}, The functional and harmonic analysis of wavelets and frames (San Antonio, TX, 1999),
Contemp. Math. {\bf 247} (1999), 149--182.
%
%

\bibitem{D}
I. Daubechies, \textit{Ten Lectures on Wavelets}, SIAM Philadelphia,
1992.

\bibitem{DJT}  J. Diestel, H. Jarchow and A. Tonge, \emph{Absolutely Summing Operators},
Cambridge University Press, Cambridge, 1995.
%
%
%
%
%

\bibitem{GH} J.-P. Gabardo and D. Han, \emph{Frames associated with measurable
spaces}, Adv. Comput. Math. {\bf 18} (2003) 127--147.

\bibitem{GH2} J.-P. Gabardo and D. Han, \emph{Frame representations for group-like unitary operator systems,}
J. Operator Theory ,  {\bf 49} (2003), 223--244.

\bibitem{HA}  D. W. Hadwin, \emph{Dilations and Hahn decompositions for linear maps}, Canad. J. Math. {\bf 33} (1981), 826--839.

\bibitem{Han1} D. Han, \emph{Frame representations and parseval duals with applications to Gabor frames,}
 Trans. Amer. Math. Soc.,  {\bf 360} (2008), 3307--3326.

\bibitem{Han2} D. Han, \emph{Dilations and completions for Gabor systems,}  J. Fourier Anal. Appl.,  {\bf 15} (2009),  201--217.

\bibitem{HL}  D. Han and D. R. Larson, \emph{Frames, bases and group
representations}, Mem. Amer. Math. Soc. {\bf 697} (2000), 1--94.

\bibitem{HLLL} D. Han, D. R. Larson, B. Liu and R. Liu, \emph{Operator-valued measures, dilations, and the theory of
frames}, Mem. Amer. Math. Soc., Vol.229, No.1075 (2014).

\bibitem{HLLL-CM} D. Han, D. R. Larson, B. Liu and R. Liu, {\em Dilations of frames, operator valued measures and bounded linear maps},
to appear in Contemp. Math.

%
%
%
%

\bibitem{KR1} Richard V. Kadison and John Ringrose, Fundamentals of the Theory of Operator Algebras, Vol. 1:  Elementary theory, Graduate Studies in Mathematics, American Mathematical Soc., 1997.


\bibitem{KR2} Richard V. Kadison and John Ringrose, Fundamentals of the Theory of Operator Algebras, Vol. 2: Advanced Theory, Graduate Studies in Mathematics, American Mathematical Soc., 1997.

\bibitem{LS} D. R. Larson and F. Szafraniec, \emph{Framings and
dilations}, Acta Sci. Math. (Szeged), \textbf{79} (2013), 529--543. 

\bibitem{M}   G. W. Mackey, \emph{Imprimitivity for representations of locally compact groups},
 Proc. Nat. Acad. Sci. U.S.A. {\bf35} (1949), 537--545.

 \bibitem{Mac-book} G. W. Mackey, Unitary group representations in physics, probability, and number theory, Addison-Wesley Pub. Co., 1989


\bibitem{Mu} G. J. Murphy, \emph{Extensions of multipliers and dilations of projective isometric representations}, Proc. Amer. Math. Soc. {\bf 125}(1) (1997), 121--127.

\bibitem{Pa}  V. Paulsen, \emph{Completely bounded maps and operator
algebras}, Cambridge University Press (2002).

%
%
%

%
%
%
%
%
%
%
%
%
%
%
%
%
%
%
%
%
%
%
%
%
%
%
%
%
%
%


\end{thebibliography}
\end{document}